\documentclass{amsart}

\usepackage{graphicx}
\usepackage{overpic}

\newtheorem{theorem}{Theorem}[section]

\newtheorem{proposition}[theorem]{Proposition}
\newtheorem{lemma}[theorem]{Lemma}

\theoremstyle{remark}
\newtheorem{remark}{Remark}[section]

\theoremstyle{definition}

\begin{document}

\title{Two-bridge knots admit no purely cosmetic surgeries}

\author{Kazuhiro Ichihara}
\address{Department of Mathematics, College of Humanities and Sciences, Nihon University, 3-25-40 Sakurajosui, Setagaya-ku, Tokyo 156-8550, JAPAN}
\email{ichihara.kazuhiro@nihon-u.ac.jp}

\author{In Dae Jong}
\address{Department of Mathematics, Kindai University, 3-4-1 Kowakae, Higashiosaka City, Osaka 577-0818, Japan} 
\email{jong@math.kindai.ac.jp}

\author{Thomas W.\ Mattman}
\address{Department of Mathematics and Statistics, California State University, Chico, Chico, CA 95929-0525}
\email{TMattman@CSUChico.edu}

\author{Toshio Saito}
\address{Department of Mathematics, Joetsu University of Education, 1 Yamayashiki, Joetsu 943-8512, JAPAN}
\email{toshio@juen.ac.jp}

\dedicatory{Dedicated to Professor Hitoshi Murakami on his 60th birthday.}



\begin{abstract}
We show that two-bridge knots and alternating fibered knots admit no purely cosmetic surgeries, i.e., no pair of distinct Dehn surgeries on such a knot produce 3-manifolds that are homeomorphic as oriented manifolds. 
Our argument, based on a recent result by Hanselman, uses several invariants of knots or 3-manifolds; 
for knots, we study the signature and some finite type invariants, and 
for 3-manifolds, we deploy the $SL(2,\mathbb{C})$ Casson invariant.
\end{abstract}


\maketitle

\section{Introduction}

A pair of Dehn surgeries are said to be \emph{purely cosmetic} 
if the two surgeries yield 3-manifolds that admit
an orientation-preserving homeomorphism. 
The Cosmetic Surgery Conjecture states that no nontrivial 
knot in $S^3$ admits purely cosmetic surgeries along inequivalent slopes. 
See \cite[Problem 1.81(A)]{Kirby} for further information and a more precise formulation.
In this paper, we confirm the conjecture for two-bridge knots and alternating fibered knots.

\begin{theorem}\label{Thm2-bridge}
Two-bridge knots admit no purely cosmetic surgeries. 
\end{theorem}

\begin{theorem}\label{ThmFibered}
Alternating fibered knots admit no purely cosmetic surgeries. 
\end{theorem}

Our argument, based on a recent result by Hanselman \cite{Hanselman}, uses 
several invariants of knots or 3-manifolds; for knots, we study the signature and some finite type invariants, and
for $3$-manifolds we deploy the $SL(2,\mathbb{C})$ Casson invariant.

\begin{remark}
We can also see that alternating pretzel knots admit no purely cosmetic surgeries as follows. 
By the result of Hanselman, an alternating pretzel knot with purely cosmetic surgeries would be of genus two and signature zero (Lemma~\ref{LemAlt}). 
An alternating pretzel knot of genus two has five strands and an odd number of crossings in each twist region. 
Then, we can diagrammatically verify that such knots are negative (i.e., all the crossings are negative crossings), up to taking the mirror image, and so, must have positive, in particular non-zero, signature by \cite{CochranGompf, Przytycki, Traczyk2}. 
Thus, such knots admit no purely cosmetic surgeries. 
\end{remark}

Let's recall some basic definitions and terminology about Dehn surgery.  
Given knot $K$ in the 3-sphere $S^3$, the following operation is
called a \textit{Dehn surgery}: 
take the exterior $E(K)$ of $K$ and glue a solid torus onto the peripheral torus $\partial E(K)$. 
The \textit{surgery slope} of the Dehn surgery on $\partial E(K)$ is represented by the curve identified with the meridian of the attached solid torus.
Using the standard meridian-longitude system, slopes on the peripheral torus are parametrized by rational numbers along with $1/0$, which corresponds to the meridian. 
When a slope $\gamma$ corresponds to a rational number $r$, Dehn surgery along $\gamma$ is called \textit{$r$-Dehn surgery}, or simply \textit{$r$-surgery}. 
The resultant manifold is denoted by $K(r)$. 

The proofs of our two theorems depend on the following lemma, due to Hanselman \cite{Hanselman}. 

\begin{lemma}\label{LemAlt}
If an alternating knot $K$ admits purely cosmetic surgeries, then the genus $g(K)$ of $K$ must be  2, the signature $\sigma(K)$ of $K$ must be 0, and the surgery slopes must be either $\pm 1$ or $\pm 2$. 
\end{lemma}

\begin{proof}
The latter two assertions follow from \cite[Theorem 3]{Hanselman} directly. 
Also from the same theorem, the Alexander polynomial of $K$ must be $\Delta_K (t) = n t^2 - 4 n t + ( 6 n + 1 ) - 4 n t^{-1} + n t^{-2}$ for some positive integer $n$. 
Then, by the work of Murasugi \cite{Murasugi58} and Crowell \cite{Crowell}, the genus $g(K)$ of $K$ must be 2. 
Note that this also follows from the results of Hanselman alone. 
Using his Theorem 2, either the genus $g(K)$ is 2 (in the case of $\pm 2$-surgeries) or else,
if the surgery is $\pm1 = \pm1/1$, we'll have
$1 \leq (t(K) + 2g)/(2g(g-1))$. 
But, since $t(K) = 0$ when $K$ is alternating, this again implies $g(K) = 2$. 
\end{proof}

\section{Two-bridge knots}

In this section, we give a proof of Theorem~\ref{Thm2-bridge} in three steps. 

Due to Lemma~\ref{LemAlt}, and since two-bridge knots are alternating, it suffices to show that, 
for a two-bridge knot $K$ of signature $\sigma(K)=0$ and genus $g(K)=2$, the surgeries
$\pm 1$ and $\pm 2$ on $K$ do not yield manifolds admitting an orientation-preserving 
homeomorphism.
 
In \cite[Corollary 4.5]{IchiharaWu}, it is shown that a 2-bridge knot of genus two admitting purely cosmetic surgeries would have the form $K_{x,y,-(x+y),x}$. 
Here, following \cite{IchiharaWu}, $K_{b_1,c_1, \cdots ,b_m,c_m}$ denotes the two-bridge knot having the Conway form $C(2 b_1, 2 c_1, \cdots , 2 b_m, 2 c_m)$. 
See Figure 3 in \cite{IchiharaWu}. 
Also note that the continued fraction associated to $K$ is
$$ [ 2x, 2y, -2(x+y), 2x] = 1/(2x+1/(2y+1/(-2(x+y)+1/2x)))\;. $$

In the following, we focus on these two-bridge knots. 

\subsection{Signature}

First, from the condition that $\sigma(K) = 0$, we have the following. 

\begin{proposition}
Let $K$ be a two-bridge knot associated to the continued fraction $[ 2x, 2y, -2(x+y), 2x] $ for integers $x>0$ and $y \ne 0$. 
If $K$ admits purely cosmetic surgeries then $y<0$ and $(x+y) > 0$. 
\end{proposition}

\begin{proof}
Consider the knot $K$ associated to the the continued fraction $[ 2x, 2y, -2(x+y), 2x]$. 
We can assume that $x > 0$ by taking the mirror image if necessary. 

There are 3 cases according to the signs of the four terms in $[ 2x, 2y, -2(x+y), 2x]$:
(i) $y>0$ and $(x+y) > 0$, (ii) $y<0$ and $(x+y) > 0$, and (iii) $y<0$ and $(x+y) < 0$. 

We use the following result of Lee \cite[Proposition 3.11]{Lee} and Traczyk \cite[Theorem 2(1)]{Traczyk} 
on the signature of an alternating knot:
for an oriented nonsplit alternating link $L$ and a reduced alternating diagram $D$ of $L$, 
\begin{equation}\label{eqSig}
\sigma(L) = o(D) - y(D) - 1
\end{equation}
holds. 
Here $\sigma(L)$ is the signature of $L$, 
$o(D)$ the number of components of the diagram obtained by 0-resolutions of pattern $A$ at all the crossings of $D$, and $y(D)$ the number of positive crossings of $D$. 

When $y>0$ and $(x+y) > 0$ (case (i)), since $[ 2x, 2y, -2(x+y), 2x] = [ 2x, 2y-1, 1, 2x+2y-2, 1, 2x-1]$, the knot $K$ has the reduced alternating diagram $D$ illustrated in Figure~\ref{fig:y>0x+y<0}. 
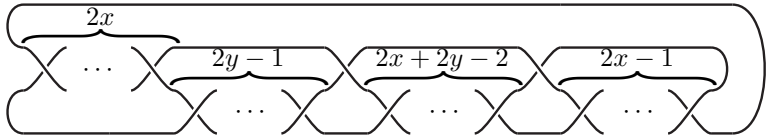
\begin{figure}[!htb]
\centering
\begin{overpic}[width=.8\textwidth]{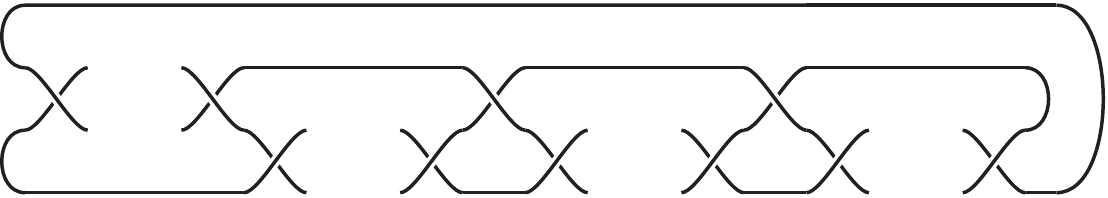}
\put(2,12){\scalebox{6}[1.5]{\rotatebox{90}{\}}}} 
\put(10.5,14.7){$2x$} 
\put(21.5,6.5){\scalebox{6}[1.5]{\rotatebox{90}{\}}}} 
\put(27,9.3){$2y-1$} 
\put(47,6.5){\scalebox{6}[1.5]{\rotatebox{90}{\}}}} 
\put(48.5,9.3){$2x+2y-2$} 
\put(72.5,6.5){\scalebox{6}[1.5]{\rotatebox{90}{\}}}} 
\put(78,9.3){$2x-1$} 
\put(10,8){$\cdots$} 
\put(30,2.5){$\cdots$} 
\put(55.5,2.5){$\cdots$} 
\put(81,2.5){$\cdots$} 
\end{overpic}
\caption{Diagram of knots $[ 2x, 2y-1, 1, 2x+2y-2, 1, 2x-1]$}\label{fig:y>0x+y<0}
\end{figure}

Then, by direct calculation, we see that $o(D) = 4x + 4y -3$ and $y(D) = 4x + 4y -2$. 
It follows that $\sigma (K) = -2$. 

In the same way, when $y<0$ and $(x+y) < 0$ (case (iii)), 
since $[ 2x, 2y, -2(x+y), 2x] = [ 2x-1, 1, -2y-2, 1, -2x-2y-1, 2x ]$, 
the knot $K$ has a reduced alternating diagram $D$ 
for which $o(D) = 2x + 3$ and $y(D) = 2x$. 
It follows that  $\sigma (K) = 2$. 

That is, in both cases, the signature, $\sigma (K)$, is non-zero. 
From Lemma~\ref{LemAlt}, this implies that if $K$ admits purely cosmetic surgeries 
we must be in the remaining case, case (ii), where $y<0$ and $(x+y) > 0$. 
\end{proof}

\begin{remark}\label{rem1}
Note that \cite[Theorem 2(1)]{Traczyk} implies Equation~\eqref{eqSig}. 
See \cite{KishimotoShibuyaTsukamoto}. 
\end{remark}

In fact, by the same calculations, we can verify that $\sigma (K) = 0$ when $y<0$ and $(x+y) > 0$ (case (ii)). 
Since $[ 2x, 2y, -2(x+y), 2x] = [ 2x-1, 1, -2y-1, 2x+2y-1, 1, 2x-1]$, the knot $K$ has a reduced alternating diagram $D$, for which $o(D) = 4x + 2y$, $y(D) = 4x + 2y -1$, so that $\sigma (K) = 0$. 

Thus, it remains to handle case (ii), where $y<0$ and $(x+y) > 0$. 
In this case, the simple continued fraction for $K$ is $[2x-1, 1, -(2y+1), 2(x+y)-1, 1, 2x-1]$. 
Note that this knot is amphichiral when the middle two terms agree, that is, when $x = -2y$.

\subsection{$SL(2,\mathbb{C})$ Casson invariant}

Using the $SL(2,\mathbb{C})$ Casson invariant, we have the following. 

\begin{proposition}\label{prp:SL2C}
Let $K$ be a two-bridge knot associated to the simple continued fraction $[2x-1, 1, -(2y+1), 2(x+y)-1, 1, 2x-1]$ for some integers $x>0$, $y<0$ with $(x+y) > 0$. 
If $K$ admits purely cosmetic surgeries then $x = -2y$. 
\end{proposition}

To prove the proposition above, our key ingredient is the $SL(2,\mathbb{C})$ Casson invariant, originally introduced by Curtis in \cite{Curtis, CurtisErratum}. 
A practical surgery formula for two-bridge knots was obtained in \cite{BodenCurtis}, and was used for a study of cosmetic surgeries on two-bridge knots in \cite{IchiharaSaito}. 

\begin{proof}[Proof of Proposition~\ref{prp:SL2C}]
Let $K$ be a two-bridge knot associated to the simple continued fraction $[2x-1, 1, -(2y+1), 2(x+y)-1, 1, 2x-1]$ for some integers $x>0$, $y<0$ with $(x+y) > 0$, and suppose that $x \ne -2y$. 

By Lemma~\ref{LemAlt}, to show that $K$ admits no purely cosmetic surgeries it suffices to prove that the pairs of 3-manifolds obtained by $\pm 1$- and $\pm 2$-surgeries on $K$ have different values of the $SL(2,\mathbb{C})$ Casson invariant.

As shown in \cite{IchiharaSaito}, based on \cite{BodenCurtis}, the following formula computes the difference in the values of the $SL(2,\mathbb{C})$ Casson invariant for the pair of surgered manifolds. 
\begin{eqnarray*}
\lambda_{SL(2, \Bbb C)} (K(p/q)) - \lambda_{SL(2, \Bbb C)} (K(-p/q)) 
&=& 
\frac{1}{4} \sum_i W_i \left( | p - q N_i |  - | -p - q N_i | \right).
\end{eqnarray*}
Here $\lambda_{SL(2, \Bbb C)} (K(p/q))$ denotes the the value of the $SL(2,\mathbb{C})$ Casson invariant of the 3-manifold obtained by $p/q$-surgery on a two-bridge knot $K$. 
Also $N_1 , \cdots , N_n$ denote the boundary slopes for $K$, and 
$W_i$ is set to be $\prod_j ( | n_j | -1 )$ for the continued fraction expansion $[c, n_1 , \cdots , n_k]$ associated to $N_i$. 
Please see \cite[Section 3]{IchiharaSaito} for details. 

Now we only consider $\pm 1$- and $\pm 2$-surgeries, i.e., $p = \pm1, \pm2$ and $q=1$. 
Moreover, since all the boundary slopes for two-bridge knots are even integers, as shown by \cite{HatcherThurston}, we see that $\lambda_{SL(2, \Bbb C)} (K(p)) - \lambda_{SL(2, \Bbb C)} (K(-p)) $ is equal to 
$$\frac{p}{2} \left(\,- \sum_{N_i >0} W_i  + \sum_{N_i <0} W_i \,\right) .$$ 

Consequently, the argument comes down to looking at $S_+  = \sum_{N_i > 0}W_i$ and $S_- = \sum_{N_i < 0}W_i$ and showing that they are not equal. 

We start with the simple continued fraction $[2x-1,1,-(2y+1),2(x+y)-1,1,2x-1]$ and calculate $S_{\pm}$  as in \cite{IchiharaSaito} based on the method originally developed in \cite[Theorem 2]{MattmanMaybrunRobinson}.

\bigskip\noindent
Case 1. $y<-1$ and $x+y>1$. 

\bigskip
We use $6$-tuples of the form $(b_1,b_2,b_3,b_4,b_5,b_6)$ with $b_j=0,1$ 
to show where substitutions are applied. 
For example, $(0,1,0,0,1,0)$ means substitution rules are applied at positions $2$ and $5$. 
Then we have the boundary slope continued fraction $[2x,2y,-2(x+y),2x]$ 
which is the longitude continued fraction. 
Hence we see that $n^+_0=2$ and  $n^-_0=2$.  

We note that each term of a boundary slope continued fractions is at least two in absolute value. 
Hence $(0,0,0,b_4,b_5,b_6)$ patterns do not give boundary slopes 
since the $1$ at position $2$ remains after making the substitutions. 
Similarly, we can eliminate $(b_1,b_2,b_3,0,0,0)$. 
We also note that the $6$-tuples can have no adjacent $1$'s.
Therefore, the following 12 cases give all boundary slope continued fractions. 

\bigskip
Case 1-1. $(0,0,1,0,0,1)$. 

After making the substitutions, we have 
$[2x-1,2,(-2,2)^{-y-1},-2(x+y),-2,(2,-2)^{x-1}]$. 

Hence $n^+_1=2x$, $n^-_1=-2y$, $N_1=4(x+y)$ and $W_1=2(x-1)\left( 2(x+y)-1 \right)$. 

\bigskip
Case 1-2. $(0,0,1,0,1,0)$. 

Then we have 
$[2x-1,2,(-2,2)^{-y-1},-2(x+y)-1,2x]$. 

Hence $n^+_2=1$, $n^-_2=1-2y$, $N_2=4y$ and $W_2=4(x-1)(2x-1)(x+y)$. 

\bigskip
Case 1-3. $(0,1,0,0,0,1)$. 

Then we have 
$[2x,2y,1-2(x+y),-2,(2,-2)^{x-1}]$. 

Hence $n^+_3=2x+1$, $n^-_3=1$, $N_3=4x$ and $W_3=-2(2x-1)(2y+1)(x+y-1)$. 

\bigskip
Case 1-4. $(0,1,0,0,1,0)$. 

Then we have 
$[2x,2y,-2(x+y),2x]$. 

Hence $n^+_4=2$, $n^-_4=2$ and $N_4=0$. 

\bigskip
Case 1-5. $(0,1,0,1,0,0)$. 

Then we have 
$[2x,2y-1,(2,-2)^{x+y-1},2,2x-1]$. 

Hence $n^+_5=2(x+y)+1$, $n^-_5=1$, $N_5=4(x+y)$ and $W_5=-4y(x-1)(2x-1)$. 

\bigskip
Case 1-6. $(0,1,0,1,0,1)$. 

Then we have 
$[2x,2y-1,(2,-2)^{x+y-1},3,(-2,2)^{x-1}]$. 

Hence $n^+_6=4x+2y-1$, $n^-_6=0$, $N_6=2(4x+2y-1)$ and $W_6=-4y(2x-1)$. 

\bigskip
Case 1-7. $(1,0,0,0,0,1)$. 

Then we have 
$[(-2,2)^{x-1},-2,2y+1,1-2(x+y),-2,(2,-2)^{x-1}]$. 

Hence $n^+_7=2x$, $n^-_7=2x$ and $N_7=0$. 

\bigskip
Case 1-8. $(1,0,0,0,1,0)$. 

Then we have 
$[(-2,2)^{x-1},-2,2y+1,-2(x+y),2x]$. 

Hence $n^+_8=1$, $n^-_8=2x+1$, $N_8=-4x$ and $W_8=-2(2x-1)(y+1)\left( 2(x+y)-1 \right)$. 

\bigskip
Case 1-9. $(1,0,0,1,0,0)$. 

Then we have 
$[(-2,2)^{x-1},-2,2y,(2,-2)^{x+y-1},2,2x-1]$. 

Hence $n^+_9=2(x+y)$, $n^-_9=2x$, $N_9=4y$ and $W_9=-2(x-1)(2y+1)$. 

\bigskip
Case 1-10. $(1,0,0,1,0,1)$. 

Then we have 
$[(-2,2)^{x-1},-2,2y,(2,-2)^{x+y-1},3,(-2,2)^{x-1}]$. 

Hence $n^+_{10}=2(2x+y-1)$, $n^-_{10}=2x-1$, $N_{10}=2\left( 2(x+y)-1 \right)$ and 

$W_{10}=-2(2y+1)$. 

\bigskip
Case 1-11. $(1,0,1,0,0,1)$. 

Then we have 
$[(-2,2)^{x-1},-3,(2,-2)^{-y-1},2(x+y),2,(-2,2)^{x-1}]$. 

Hence $n^+_{11}=2x-1$, $n^-_{11}=2(x-y-1)$, $N_{11}=2(2y+1)$ and 
$W_{11}=2\left( 2(x+y)-1 \right)$. 

\bigskip
Case 1-12. $(1,0,1,0,1,0)$. 

Then we have 
$[(-2,2)^{x-1},-3,(2,-2)^{-y-1},2(x+y)+1,-2x]$. 

Hence $n^+_{12}=0$, $n^-_{12}=2(x-y)-1$, $N_{12}=-2\left( 2(x-y)-1 \right)$ and 

$W_{12}=4(x+y)(2x-1)$. 

\bigskip\noindent
Case 2. $y<-1$ and $x+y=1$. 

\bigskip
As in Case 1, we can eliminate the patterns $(b_1,b_2,b_3,0,0,0)$ and $(b_1,b_2,0,0,0,b_6)$. 
Therefore, the following 10 cases give all boundary slope continued fractions. 

\bigskip
Case 2-1. $(0,0,1,0,0,1)$. 

Then we have 
$[2x-1,2,(-2,2)^{x-2},-2,-2,(2,-2)^{x-1}]$. 

Hence $n^+_1=2x$, $n^-_1=2x-2$, $N_1=4$ and $W_1=2(x-1)$. 

\bigskip
Case 2-2. $(0,0,1,0,1,0)$. 

Then we have 
$[2x-1,2,(-2,2)^{x-2},-3,2x]$. 

Hence $n^+_2=1$, $n^-_2=2x-1$, $N_2=-4(x-1)$ and $W_2=4(x-1)(2x-1)$. 

\bigskip
Case 2-3. $(0,1,0,0,1,0)$. 

Then we have 
$[2x,-2(x-1),-2,2x]$. 

Hence $n^+_3=2$, $n^-_3=2$ and $N_3=0$. 

\bigskip
Case 2-4. $(0,1,0,1,0,0)$. 

Then we have 
$[2x,-2x+1,2,2x-1]$. 

Hence $n^+_4=3$, $n^-_4=1$, $N_4=4$ and $W_4=4(x-1)^2(2x-1)$. 

\bigskip
Case 2-5. $(0,1,0,1,0,1)$. 

Then we have 
$[2x,-2x+1,3,(-2,2)^{x-1}]$. 

Hence $n^+_5=2x+1$, $n^-_5=0$, $N_5=2(2x+1)$ and $W_5=4(x-1)(2x-1)$. 

\bigskip
Case 2-6. $(1,0,0,0,1,0)$. 

Then we have 
$[(-2,2)^{x-1},-2,-2x+3,-2,2x]$. 

Hence $n^+_6=1$, $n^-_6=2x+1$, $N_6=-4x$ and $W_6=2(x-2)(2x-1)$. 

\bigskip
Case 2-7. $(1,0,0,1,0,0)$. 

Then we have 
$[(-2,2)^{x-1},-2,-2(x-1),2,2x-1]$. 

Hence $n^+_7=2$, $n^-_7=2x$, $N_7=-4(x-1)$ and $W_7=2(x-1)(2x-3)$. 

\bigskip
Case 2-8. $(1,0,0,1,0,1)$. 

Then we have 
$[(-2,2)^{x-1},-2,-2(x-1),3,(-2,2)^{x-1}]$. 

Hence $n^+_8=2x$, $n^-_8=2x-1$, $N_8=2$ and $W_8=2(2x-3)$. 

\bigskip
Case 2-9. $(1,0,1,0,0,1)$. 

Then we have 
$[(-2,2)^{x-1},-3,(2,-2)^{x-2},2,2,(-2,2)^{x-1}]$. 

Hence $n^+_9=2x-1$, $n^-_9=4(x-1)$, $N_9=-2(2x-3)$ and $W_9=2$. 

\bigskip
Case 2-10. $(1,0,1,0,1,0)$. 

Then we have 
$[(-2,2)^{x-1},-3,(2,-2)^{x-2},3,-2x]$. 

Hence $n^+_{10}=0$, $n^-_{10}=4x-3$, $N_{10}=-2(4x-3)$ and $W_{10}=4(2x-1)$. 

\bigskip\noindent
Case 3. $y=-1$ and $x+y>1$ (i.e., $x>2$). 

\bigskip
We can eliminate $(b_1,b_2,b_3,0,0,0)$ and $(b_1,0,0,0,b_5,b_6)$. 
Therefore, the following 10 cases give all the boundary slope continued fractions. 

\bigskip
Case 3-1. $(0,0,1,0,0,1)$. 

Then we have 
$[2x-1,2,-2(x-1),-2,(2,-2)^{x-1}]$. 

Hence $n^+_1=2x$, $n^-_1=2$, $N_1=4(x-1)$ and $W_1=2(x-1)(2x+1)$. 

\bigskip
Case 3-2. $(0,0,1,0,1,0)$. 

Then we have 
$[2x-1,2,-2x+1,2x]$. 

Hence $n^+_2=1$, $n^-_2=3$, $N_2=-4$ and $W_2=4(x-1)^2(2x-1)$. 

\bigskip
Case 3-3. $(0,1,0,0,0,1)$. 

Then we have 
$[2x,-2,-2x+3,-2,(2,-2)^{x-1}]$. 

Hence $n^+_3=2x+1$, $n^-_3=1$, $N_3=4x$ and $W_3=2(2x-1)(x-2)$. 

\bigskip
Case 3-4. $(0,1,0,0,1,0)$. 

Then we have 
$[2x,-2,-2(x-1),2x]$. 

Hence $n^+_4=2$, $n^-_4=2$ and $N_4=0$. 

\bigskip
Case 3-5. $(0,1,0,1,0,0)$. 

Then we have 
$[2x,-3,(2,-2)^{x-2},2,2x-1]$. 

Hence $n^+_5=2x-1$, $n^-_5=1$, $N_5=4(x-1)$ and $W_5=4(x-1)(2x-1)$. 

\bigskip
Case 3-6. $(0,1,0,1,0,1)$. 

Then we have 
$[2x,-3,(2,-2)^{x-2},3,(-2,2)^{x-1}]$. 

Hence $n^+_6=4x-3$, $n^-_6=0$, $N_6=2(4x-3)$ and $W_6=4(2x-1)$. 

\bigskip
Case 3-7. $(1,0,0,1,0,0)$. 

Then we have 
$[(-2,2)^{x-1},-2,-2,(2,-2)^{x-2},2,2x-1]$. 

Hence $n^+_7=2(x-1)$, $n^-_7=2x$, $N_7=-4$ and $W_7=2(x-1)$. 

\bigskip
Case 3-8. $(1,0,0,1,0,1)$. 

Then we have 
$[(-2,2)^{x-1},-2,-2,(2,-2)^{x-2},3,(-2,2)^{x-1}]$. 

Hence $n^+_8=4(x-1)$, $n^-_8=2x-1$, $N_8=2(2x-3)$ and $W_8=2$. 

\bigskip
Case 3-9. $(1,0,1,0,0,1)$. 

Then we have 
$[(-2,2)^{x-1},-3,2(x-1),2,(-2,2)^{x-1}]$. 

Hence $n^+_9=2x-1$, $n^-_9=2x$, $N_9=-2$ and 
$W_9=2(2x-3)$. 

\bigskip
Case 3-10. $(1,0,1,0,1,0)$. 

Then we have 
$[(-2,2)^{x-1},-3,2x-1,-2x]$. 

Hence $n^+_{10}=0$, $n^-_{10}=2x+1$, $N_{10}=-2(2x+1)$ and $W_{10}=4(x-1)(2x-1)$. 

\bigskip\noindent
Case 4. $y=-1$ and $x+y=1$ (i.e., $x=2$). 

\bigskip
We can eliminate $(b_1,b_2,b_3,0,0,0)$, $(b_1,b_2,0,0,0,b_6)$, and $(b_1,0,0,0,b_5,b_6)$. 
This leaves 9 cases in order to obtain all the boundary slope continued fractions. 

\bigskip
Case 4-1. $(0,0,1,0,0,1)$. 

Then we have 
$[3,2,-2,-2,2,-2]$. 

Hence $n^+_1=4$, $n^-_1=2$, $N_1=4$ and $W_1=2$. 

\bigskip
Case 4-2. $(0,0,1,0,1,0)$. 

Then we have 
$[3,2,-3,4]$. 

Hence $n^+_2=1$, $n^-_2=3$, $N_2=-4$ and $W_2=12$. 

\bigskip
Case 4-3. $(0,1,0,0,1,0)$. 

Then we have 
$[4,-2,-2,4]$. 

Hence $n^+_3=2$, $n^-_3=2$ and $N_3=0$. 

\bigskip
Case 4-4. $(0,1,0,1,0,0)$. 

Then we have 
$[4,-3,2,3]$. 

Hence $n^+_4=3$, $n^-_4=1$, $N_4=4$ and $W_4=12$. 

\bigskip
Case 4-5. $(0,1,0,1,0,1)$. 

Then we have 
$[4,-3,3,-2,2]$. 

Hence $n^+_5=5$, $n^-_5=0$, $N_5=10$ and $W_5=12$. 

\bigskip
Case 4-6. $(1,0,0,1,0,0)$. 

Then we have 
$[-2,2,-2,-2,2,3]$. 

Hence $n^+_6=2$, $n^-_6=4$, $N_6=-4$ and $W_6=2$. 

\bigskip
Case 4-7. $(1,0,0,1,0,1)$. 

Then we have 
$[-2,2,-2,-2,3,-2,2]$. 

Hence $n^+_7=4$, $n^-_7=3$, $N_7=-2$ and $W_7=2$. 

\bigskip
Case 4-8. $(1,0,1,0,0,1)$. 

Then we have 
$[-2,2,-3,2,2,-2,2]$. 

Hence $n^+_8=3$, $n^-_8=4$, $N_8=-2$ and $W_8=2$. 

\bigskip
Case 4-9. $(1,0,1,0,1,0)$. 

Then we have 
$[-2,2,-3,3,-4]$. 

Hence $n^+_9=0$, $n^-_9=5$, $N_9=-10$ and $W_9=12$. 

\bigskip

Combining these calculations, we have
$$S_+ = -2(2y+1)-4y(2x-1)+2(x-1)(2(x+y)-1)+2(2x-1)(-(2y+1))(x+y-1)-4y(2x-1)(x-1),$$
$$S_- = 2(2(x+y)-1) + 4(x+y)(2x-1)-2(y+1)(2(x+y)-1)(2x-1)-2(2y+1)(x-1)+4(x-1)(x+y)(2x-1).$$

Thus, we obtain $S_- - S_+ = 2(x+2y)(4x^2-6x+5)$. 
If we are not in the amphichiral case, that is, if $x \ne -2y$, this difference is positive as $4x^2-6x+5$ has imaginary roots. 

As shown in \cite{IchiharaSaito}, this implies that if $K$ admits purely cosmetic surgeries we must have $x = -2y$. 
\end{proof}

Note that if $x = -2y$, then the knot $K$ is associated to the continued fraction $[4n,-2n,-2n,4n]$ for a positive integer $n$

\subsection{Finite type invariants}

Finally, by using finite type invariants of knots, we have the following. 

\begin{proposition}\label{clm:FTI} 
The two-bridge knot $K$ associated to the continued fraction $[4n,-2n,-2n,4n]$ for a positive integer $n$ admits no purely cosmetic surgeries. 
\end{proposition}

For a knot $K$, let $a_{2m}(K)$ be the coefficient of $z^{2m}$ in 
the Conway polynomial $\nabla_K(z)$ of $K$. 
As an obstruction to cosmetic surgery, 
Boyer and Lines showed the following. 

\begin{proposition}[{\cite[Proposition 5.1]{BoyerLines}}]\label{prop:a2} 
If a knot $K$ admits purely cosmetic surgery, 
then $a_2(K) = 0$. 
\end{proposition} 

Calculating the Conway polynomial, we have the following. 

\begin{lemma}\label{lem:Conway} 
For the two-bridge knot 
$K = C[4n,-2n,-2n,4n]$ with $n >0$, 
\[ \nabla_K(z) = 1 + 4n^4 z^4 \, . \]
\end{lemma} 
\begin{proof} 
We omit the proof since this can be calculated easily by hand. 
This also can be confirmed by Hanselman's result about the constraints on the Alexander polynomial in \cite[Theorem 3]{Hanselman}. 
\end{proof} 

Using finite type invariants, Ito \cite{ItoTetsuya} proposed the following obstruction to cosmetic surgery.

\begin{proposition}[{\cite[Corollary 1.5 (i)]{ItoTetsuya}}] 
Let $K$ be a knot 
and $r = p/q \in \mathbb{Q} \setminus \{0\}$. 
If $K(r) \cong K(r')$ for $r' \ne r$, then 
\begin{equation} 
p^2 \left( 24 w_4(K) - 5 v_4(K) \right) 
+ 5 v_4(K) + q^2\left( 210 v_6(K) + 5 v_4(K) \right) = 0 \, .  \label{eq:Ito}
\end{equation}  
\end{proposition} 
Here $v_4(K)$, $w_4(K)$, and $v_6(K)$ are certain canonical finite type invariants
of the knot $K$, 
which are determined by the Conway polynomial and the Jones polynomial. 
Since we have 
$a_2(K) = 0$,  $a_4(K) = 4n^4$, and $a_6(K) = 0$ by Lemma~\ref{lem:Conway}, 
these invariants are as follows (see \cite[Lemma 2.1]{ItoTetsuya}): 
\begin{itemize}
\item 
$v_4 (K) = -\frac12 a_4(K) = -2n^4$. 
\item 
$w_4(K) = 
\frac{1}{96} j_4(K) + \frac{3}{32} a_4(K) = \frac{1}{96} j_4(K) + \frac{3}{8} n^4$. 
\item 
$v_6(K) = -\frac{1}{3} n^4$. 
\end{itemize} 
Here $j_4(K)$ is the coefficient of $h^4$ in the Jones polynomial $V_K(e^h)$ of $K$, 
using the variable $t = e^h$. 
Thus equation \eqref{eq:Ito} reduces to 
\[ p^2\left( \frac{1}{4}j_4(K) + 19n^4 \right) - 10n^4 - 80 q^2 n^4 =0 \, . \]

Furthermore, since we may assume that $p/q = \pm 1, \pm 2$, 
we have $(p^2, q^2) = (1,1)$ or $(4,1)$. 
Thus, to prove Proposition~\ref{clm:FTI}, 
it suffices to show that 
\begin{align} 
j_4(K) \ne 14n^4 \ \text { and } \  j_4(K) \ne 284n^4 \ \text{ for } \ n > 0. \label{eq:j4}
\end{align} 

On the other hand, we have the following. 

\begin{lemma}\label{lem:Jones} 
For the two-bridge knot $K = C[4n,-2n,-2n,4n]$ with $n > 0$, 
\[ j_4(K) = -12n^4 \, . \]
\end{lemma} 
\begin{proof} 
Recall the skein relation of the Jones polynomial $V_K(t)$; 
\[ t^{-1} V_{ 
    \begin{minipage}{10pt}
        \begin{picture}(10,10)
            \put(0,0){\vector(1,1){10}}
            \qbezier(10,0)(10,0)(6.6667,3.3334)
            \qbezier(3.3334,6.6667)(0,10)(0,10)
            \put(0,10){\vector(-1,1){0}}
        \end{picture}
    \end{minipage}}
(t) -t V_{
    \begin{minipage}{10pt}
        \begin{picture}(10,10)
						\put(10,0){\vector(-1,1){10}}
            \qbezier(0,0)(0,0)(3.3334,3.3334)
            \qbezier(6.6667,6.6667)(10,10)(10,10)
            \put(10,10){\vector(1,1){0}}
        \end{picture}
    \end{minipage}}
(t) = (t^{\frac{1}{2}}-t^{-\frac{1}{2}}) V_{
    \begin{minipage}{10pt}
        \begin{picture}(10,10)
            \qbezier(0,0)(6.6667,5)(0,10)
            \qbezier(10,0)(3.3334,5)(10,10)
            \put(10,10){\vector(1,1){0}}
            \put(0,10){\vector(-1,1){0}}
        \end{picture}
    \end{minipage}}(t) \, .\] 
This is equivalent to one of the following: 
\begin{align} 
V_{ 
    \begin{minipage}{10pt}
        \begin{picture}(10,10)
            \put(0,0){\vector(1,1){10}}
            \qbezier(10,0)(10,0)(6.6667,3.3334)
            \qbezier(3.3334,6.6667)(0,10)(0,10)
            \put(0,10){\vector(-1,1){0}}
        \end{picture}
    \end{minipage}}
(t) &= t^2 V_{
    \begin{minipage}{10pt}
        \begin{picture}(10,10)
						\put(10,0){\vector(-1,1){10}}
            \qbezier(0,0)(0,0)(3.3334,3.3334)
            \qbezier(6.6667,6.6667)(10,10)(10,10)
            \put(10,10){\vector(1,1){0}}
        \end{picture}
    \end{minipage}}
(t)+  t (t^{\frac{1}{2}}-t^{-\frac{1}{2}}) V_{
    \begin{minipage}{10pt}
        \begin{picture}(10,10)
            \qbezier(0,0)(6.6667,5)(0,10)
            \qbezier(10,0)(3.3334,5)(10,10)
            \put(10,10){\vector(1,1){0}}
            \put(0,10){\vector(-1,1){0}}
        \end{picture}
    \end{minipage}}(t) \, , \label{eq:posi} \\ 
V_{
    \begin{minipage}{10pt}
        \begin{picture}(10,10)
						\put(10,0){\vector(-1,1){10}}
            \qbezier(0,0)(0,0)(3.3334,3.3334)
            \qbezier(6.6667,6.6667)(10,10)(10,10)
            \put(10,10){\vector(1,1){0}}
        \end{picture}
    \end{minipage}}
(t) &= t^{-2} V_{ 
    \begin{minipage}{10pt}
        \begin{picture}(10,10)
            \put(0,0){\vector(1,1){10}}
            \qbezier(10,0)(10,0)(6.6667,3.3334)
            \qbezier(3.3334,6.6667)(0,10)(0,10)
            \put(0,10){\vector(-1,1){0}}
        \end{picture}
    \end{minipage}}
(t)  - t^{-1} (t^{\frac{1}{2}}-t^{-\frac{1}{2}}) V_{
    \begin{minipage}{10pt}
        \begin{picture}(10,10)
            \qbezier(0,0)(6.6667,5)(0,10)
            \qbezier(10,0)(3.3334,5)(10,10)
            \put(10,10){\vector(1,1){0}}
            \put(0,10){\vector(-1,1){0}}
        \end{picture}
    \end{minipage}}(t) \, . \label{eq:nega}
\end{align}
Applying the skein relation \eqref{eq:posi} to the marked positive crossing 
in Figure~\ref{fig:c4224}, we have 
\[ V_K (t) = t^2 V_{C[4n, -2n, -2n+2, 4n]} (t) + t ( t^{\frac{1}{2}} - t^{-\frac{1}{2}}) V_{L_n}(t) \, , \]
where the link $L_n$ is the link of Figure~\ref{fig:Ln}, 
which is the connected sum of $C[4n, -2n]$ 
and the torus link $T(2,-4n)$ with coherent orientations. 
Repeating this procedure $n$ times, we have 
\begin{align*} 
V_K (t) 
&= t^{2n} V_{C[4n, 2n]} (t) 
+ t (1 + t^2 + \cdots + t^{2(n-1)}) ( t^{\frac{1}{2}} - t^{-\frac{1}{2}}) V_{L_n}(t) \\ 
&= t^{2n} V_{C[4n, 2n]} (t) 
+ t \dfrac{1-t^{2n}}{1-t^2} ( t^{\frac{1}{2}} - t^{-\frac{1}{2}}) V_{L_n}(t) \\ 
&= t^{2n} V_{C[4n, 2n]} (t) 
- t^{\frac12} \dfrac{1-t^{2n}}{1+t} V_{C[4n,-2n]}(t) V_{T(2,-4n)}(t)\, . 
\end{align*} 
The last equality relies on the Jones polynomial of a connected sum
being the product of those of the factors.

\begin{figure}[!htb]
\centering
\begin{overpic}[width=.8\textwidth]{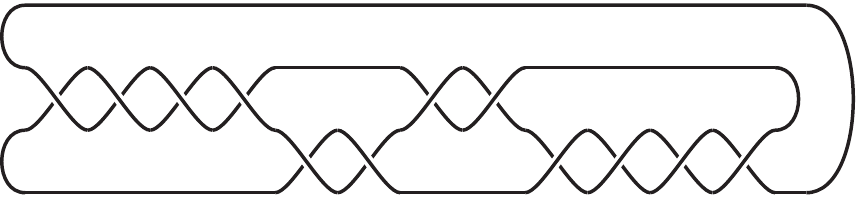}
\put(53.2,10){*} 
\put(5,15){\scalebox{7}[1.5]{\rotatebox{90}{\}}}} 
\put(64,8){\scalebox{7}[1.5]{\rotatebox{90}{\}}}} 
\put(49,15){\scalebox{3}[1.5]{\rotatebox{90}{\}}}} 
\put(34,8){\scalebox{3}[1.5]{\rotatebox{90}{\}}}} 
\put(15.7,18){$4n$}
\put(74.5,11){$4n$}
\put(52.8,18){$2n$}
\put(38,11){$2n$}
\end{overpic}
\caption{Apply the skein relation at the positive crossing marked by $*$.} 
\label{fig:c4224}
\end{figure}

\begin{figure}[!htb]
\centering
\begin{overpic}[width=.8\textwidth]{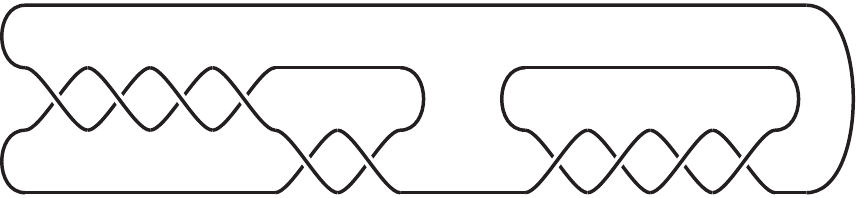}
\put(5,15){\scalebox{7}[1.5]{\rotatebox{90}{\}}}} 
\put(64,8){\scalebox{7}[1.5]{\rotatebox{90}{\}}}} 
\put(34,8){\scalebox{3}[1.5]{\rotatebox{90}{\}}}} 
\put(15.7,18){$4n$}
\put(74.5,11){$4n$}
\put(38,11){$2n$}
\end{overpic}
\caption{The link $L_n$ is the connected sum of the knot $C[4n,-2n]$ and $T(2,-4n)$.} 
\label{fig:Ln}
\end{figure}
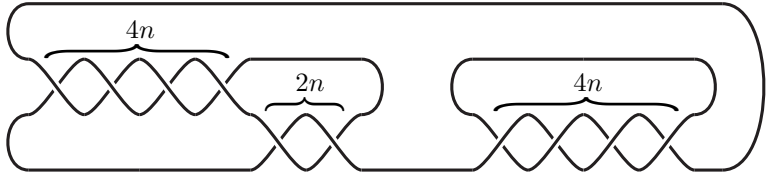

For $T(2,-4n)$, applying the skein relation \eqref{eq:posi} $2n$ times, we have 
\begin{align*} 
V_{T(2,-4n)}(t) 
&= t^2 V_{T(2,-4n+2)}(t) + t (t^{\frac12} - t^{-\frac12}) \\ 
&= \cdots \\ 
&= t^{4n} V_{\circ \circ} + (1 + t^2 + \cdots + t^{4n-2}) t (t^{\frac12} - t^{-\frac12}) \\ 
&= t^{4n} (-t^{\frac12} - t^{-\frac12}) + \dfrac{1 - t^{4n}}{1-t^2} t (t^{\frac12} - t^{-\frac12}) \\ 
&= \dfrac{-t^{\frac12}}{1+t} \left( t^{4n}(1 + t^{-1})(1+t) + 1 - t^{4n} \right) \\ 
&= \dfrac{-t^{\frac12}}{1+t} \left( t^{4n}(t + 1 + t^{-1}) + 1 \right) \, , 
\end{align*} 
where $V_{\circ \circ} = -t^{\frac12} - t^{-\frac12}$ is 
the Jones polynomial of the trivial 2-component link. 

For $C[4n, -2n]$, applying the skein relation \eqref{eq:nega} $n$ times, we have 
\begin{align*} 
V_{C[4n, -2n]}(t) 
&= t^{-2} V_{C[4n, -2n+2]}(t) - t^{-1} (t^{\frac12} - t^{-\frac12}) V_{T(2,4n)} \\ 
&= \cdots \\ 
&= t^{-2n} - (1 + t^{-2} + \cdots + t^{-2(n-1)}) t^{-1} (t^{\frac12} - t^{-\frac12}) V_{T(2,4n)} \\ 
&= t^{-2n} - \dfrac{1 - t^{-2n}}{1-t^{-2}} t^{-1} (t^{\frac12} - t^{-\frac12}) 
\cdot \dfrac{-t^{-\frac12}}{1+t^{-1}} \left( t^{-4n}(t + 1 + t^{-1}) + 1 \right) \\ 
&= t^{-2n} + \dfrac{1 - t^{-2n}}{(1+t)(1 + t^{-1})}  \left( t^{-4n}(t + 1 + t^{-1}) + 1 \right) \, . 
\end{align*} 

Similarly, for $C[4n, 2n]$, applying the skein relation \eqref{eq:posi} $n$ times, we have 
\begin{align*} 
V_{C[4n, 2n]}(t) 
&= t^{2} V_{C[4n, 2n-2]}(t) + t (t^{\frac12} - t^{-\frac12}) V_{T(2,4n)} \\ 
&= \cdots \\ 
&= t^{2n} + \dfrac{1 - t^{2n}}{1-t^{2}} t (t^{\frac12} - t^{-\frac12}) 
\cdot \dfrac{-t^{-\frac12}}{1+t^{-1}} \left( t^{-4n}(t + 1 + t^{-1}) + 1 \right) \\ 
&= t^{2n} + \dfrac{1 - t^{2n}}{(1+t)(1 + t^{-1})}  \left( t^{-4n}(t + 1 + t^{-1}) + 1 \right) \, . 
\end{align*} 

Therefore we have 
\begin{align*} 
V_K (t) 
=& t^{2n} \left( t^{2n} + \dfrac{1 - t^{2n}}{(1+t)(1 + t^{-1})}  \left( t^{-4n}(t + 1 + t^{-1}) + 1 \right) \right) \\ 
&- t^{\frac12} \dfrac{1-t^{2n}}{1+t} 
\left(t^{-2n} + \dfrac{1 - t^{-2n}}{(1+t)(1 + t^{-1})}  \left( t^{-4n}(t + 1 + t^{-1}) + 1 \right) \right) 
 \dfrac{-t^{\frac12}}{1+t} \left( t^{4n}(t + 1 + t^{-1}) + 1 \right) \\ 
=&  
t^{4n} + \dfrac{1 - t^{2n}}{(1+t)(1 + t^{-1})}  \left( t^{-2n}(t + 1 + t^{-1}) + t^{2n} \right)  \\ 
&+ \dfrac{1-t^{2n}}{(1+t)(1+t^{-1})} \left( t^{2n}(t + 1 + t^{-1}) + t^{-2n}  \right) \\
&+ \dfrac{(1-t^{2n})(1 - t^{-2n})}{(1+t)^2(1+t^{-1})^2} \left( t^{-4n}(t + 1 + t^{-1}) + 1 \right) 
 \left( t^{4n}(t + 1 + t^{-1}) + 1 \right) \\ 
=&  
t^{2n} + t^{-2n} -1+ (1 - t^{2n}) (1 - t^{-2n}) \left(1 +  \dfrac{(t^{2 n} - t^{-2 n})^2}{(1 + t) (1 + t^{-1})} 
+ \dfrac{(1 - t^{4 n})(1 - t^{-4 n})}{(1 + t)^2 (1 + t^{-1})^2} \right) \\
=&  
1 + (1 - t^{2n}) (1 - t^{-2n}) \left(\dfrac{(t^{2 n} - t^{-2 n})^2}{(1 + t) (1 + t^{-1})} 
+ \dfrac{(1 - t^{4 n})(1 - t^{-4 n})}{(1 + t)^2 (1 + t^{-1})^2} \right) \\ 
=&  
1 + \dfrac{(1 - t^{2n}) (1 - t^{-2n}) (t^{2n} - t^{-2n})^2 (t+1+t^{-1})}{(1 + t)^2 (1 + t^{-1})^2} \, . 
\end{align*} 


Using Mathematica \cite{Mathematica}, we verify that $j_4(K) = -12 n^4$. 
\end{proof}

Lemma~\ref{lem:Jones} guarantees that condition \eqref{eq:j4} holds, 
which completes the proof of Proposition~\ref{clm:FTI} and with it the proof of Theorem~\ref{Thm2-bridge}.

\section{Alternating fibered knots}

In this section we prove Theorem~\ref{ThmFibered}. 
By Lemma~\ref{LemAlt}, for an alternating fibered knot to admit purely cosmetic surgeries, it must have signature zero  and genus two. 

Due to Stoimenow's result \cite[Proposition 3.2]{Stoimenow-g2}, a prime alternating genus two knot has zero signature if and only if its diagram can be obtained from a diagram of $6_3$, $7_7$, $8_{12}$, $9_{41}$, $10_{58}$ or $12_{1202}$ by (repeated) $\bar{t'_2}$ moves. 
See \cite[Definition 2.2]{Stoimenow-g2} for the definition of the $\bar{t'_2}$ move. 
In the same way, non-prime alternating knots are obtained from $3_1 \sharp 3_1^*$, $4_1 \sharp 4_1$ by \cite[Corollary 2.3]{Stoimenow} together with \cite[Theorem 1(b)]{Menasco}.

On the other hand, an alternating knot is fibered if and only if the Alexander polynomial is monic, as proved by Murasugi \cite{Murasugi63}, using a result of Neuwirth \cite{Neuwirth}. 
However, as shown by the second author \cite[Corollary 4.6]{Jong}, if one applies a $\bar{t'_2}$ move to an alternating diagram, then (the absolute value of) the coefficients of the Alexander polynomial must increase. 

Thus, alternating fibered knots of genus 2 are exactly the fibered knots contained in $6_3$, $7_7$, $8_{12}$, $9_{41}$, $10_{58}$ or $12_{1202}$, or $3_1 \sharp 3_1^*$, $4_1 \sharp 4_1$. 
Here $\sharp$ denotes the connected sum and $*$ the mirror image.

In fact, the fibered ones are $6_3$, $7_7$, $8_{12}$, $3_1 \sharp 3_1^*$, and $4_1 \sharp 4_1$. 
However, the Alexander polynomials of these knots do not satisfy Hanselman's condition in \cite[Theorem 3]{Hanselman}. 
Also see \cite[Corollary 5.2]{Jong}. 

This completes the proof of Theorem~\ref{ThmFibered}.

\section*{Acknowledgements}
The authors would like to thank Hitoshi Murakami and Ilya Kofman for giving them information about Remark~\ref{rem1}, and also Tatsuya Tsukamoto for a useful conversation about it. 
Ichihara is partially supported by JSPS KAKENHI Grant Number 18K03287. 
Jong is partially supported by JSPS KAKENHI Grant Number 19K03483. 
Saito is partially supported by JSPS KAKENHI Grant Number 15K04869.

\end{document}